\documentclass[11pt]{amsart}
\usepackage[paper,pkg/amsaddr=true]{zbs}
\title{A Hal\'asz-type theorem for permutation anticoncentration}
\author[Zach Hunter]{Zach Hunter\nfts{1}}\address{\nfts{1}Department of Mathematics, ETH Zurich, Zurich, Switzerland}
\email{zach.hunter@math.ethz.ch}
\author[Cosmin Pohoata]{Cosmin Pohoata\nfts{2}}\address{\nfts{2}Department of Mathematics, Emory University, Atlanta, USA}
\email{cosmin.pohoata@emory.edu}
\author[Daniel G.\ Zhu]{Daniel G.\ Zhu\nfts{3}}\address{\nfts{3}Department of Mathematics, Princeton University, Princeton, USA}
\email{zhd@princeton.edu}
\begin{document}
\begin{abstract}
Given a set $A=\{a_1,\ldots,a_n\}$ of real numbers and real coefficients $b_1,\ldots,b_n$, consider the distribution of the sum obtained by pairing the $a_i$'s with the $b_i$'s according to a uniformly random permutation.  A recent theorem of Pawlowski shows that as soon as the coefficients are not all equal, this distribution is always spread out at scale $n^{-1}$: no single value can occur with probability larger than $\frac{1}{2\lceil n/2\rceil + 1}$, and this bound is sharp in general.

We show that stronger anticoncentration holds when the coefficients have additional diversity.  We quantify the structure of the coefficient multiset by a simple statistic depending on its multiplicity profile, and prove that the maximum point mass of the permuted sum decays polynomially faster as this statistic grows. In particular, when the coefficients are all distinct we obtain a bound of $n^{-5/2+o(1)}$, which can be regarded as an analogue of a classical theorem of Erd\H{o}s and Moser.
\end{abstract}
\maketitle

\section{Introduction} \label{sec:intro}

Let $A = \set{a_1,\ldots,a_n}$ is a set of $n$ real numbers and let $S$ be the set of points $(a_{\pi(1)}, \ldots, a_{\pi(n)})$, where $\pi$ ranges over all permutations of $[n] = \set{1,2,\ldots,n}$. What is the maximum number of points in $S$ that lie in a hyperplane? As stated, this question is trivial, as every point in $S$ lies in
\[H_0 = \setmid*{(x_1,\ldots,x_n) \in \setr^n}{x_1 + \cdots + x_n = \sum_{a \in A} a}.\]
If we exclude $H_0$, this question was recently answered by Pawlowski \cite{Pawlo24}, who showed that the maximum number of points in $S$ lying on a non-$H_0$ hyperplane is $(n-1)!$ if $n\geq 3$ is odd and $n\cdot (n-2)!$ if $n \geq 4$ is even. In the case $A = [n]$, this result can be interpreted as bounding the maximum number of vertices of a permutohedron that can lie on a hyperplane, answering a question of Heged\"us and K\'arolyi \cite{HK24}. Recently, Kong and Zeng \cite{KZ25} proved a similar result over the complex numbers, which in turn extends to all characteristic-zero fields.

A close examination of \cite{Pawlo24} shows that for $n \geq 8$, the only hyperplanes achieving equality in the above bound are of the form $b_1x_1 + \cdots + b_nx_n = c$ where at least $n-2$ of the $b_i$ are the same. Combined with the fact that $H_0$ contains all of $S$, this suggests that as the $b_i$ become more ``diverse'', it should be possible to prove stronger upper bounds on the number of points in $S$ lying on the hyperplane. In this note, we prove a result of this form that is optimal up to a factor of $n^{o(1)}$.

For the remainder of this note, we will use the language of anticoncentration, where for a random variable $X$ we let $Q[X] = \sup_x \setp[X = x]$. The result of Pawlowski \cite{Pawlo24} can be written as follows:
\begin{thm}[Pawlowski] \label{thm:paw}
Let $A = \set{a_1,\ldots,a_n}$ be a set of $n \geq 3$ real numbers and let $b_1,\ldots,b_n$ be real coefficients, not all the same. If $\pi$ is a uniformly random permutation of $[n]$, then
\[Q\brac*{\sum_{i=1}^n a_i b_{\pi(i)}} \leq \frac{1}{2\ceil{n/2} + 1}.\]
\end{thm}

To state our result, for a multiset $B = \set{b_1,\ldots,b_n}$ we define its ``multiplicity profile'' $\mu(B) = (\mu_1,\ldots,\mu_\ell)$ to be the partition of $n$ such that the distinct elements in $B$ occur with multiplicities $\mu_1 \geq \cdots \geq \mu_\ell$ (so $\ell$ is the number of distinct elements in $B$). We additionally define $M(B) = \sum_{i=1}^\ell (i-1)^2 \mu_i$.

\begin{thm} \label{thm:main}
Let $A = \set{a_1,\ldots,a_n}$ and $B = \set{b_1,\ldots,b_n}$ be a set and a multiset, respectively, of real numbers such that $M(B) > 0$. If $\pi$ is a uniformly random permutation of $[n]$, then
\[Q\brac*{\sum_{i=1}^n a_i b_{\pi(i)}} \leq \frac{1}{n^{1-o(1)} \sqrt{M(B)}}.\]
\end{thm}

\cref{thm:main} can be interpreted as a permutation analogue of anticoncentration results of Erd\H{o}s-Moser, S\'ark\"ozy-Szemer\'edi, Stanley, and Hal\'asz. For context, we recall the classical problem of Littlewood and Offord \cite{LittlewoodOfford43}, which asked for a upper bound on $Q[X]$ where
$X=\sum_{i=1}^n a_i\xi_i$ with independent $\xi_i\in\{0,1\}$. In 1945, Erd\H{o}s proved the sharp universal bound $Q[X]\lsim n^{-1/2}$ \cite{Erdos45}. It was later proven by Erd\H{o}s and Moser \cite{ErdosMoser} that if the $a_i$ are pairwise distinct, this bound can be improved to $Q[X] \lsim n^{-3/2+o(1)}$. Answering a question of Erd\H{o}s and Moser \cite{ErdosMoser}, S\'ark\"ozy-Szemer\'edi \cite{sarkozy1965uber} and then also Stanley \cite{Stan80} further established the optimal bound $Q[X] \lsim n^{-3/2}$ in this case (with different levels of precision and substantially different proof methods). Hal\'asz \cite{Hal77} later
strengthened this theory by showing that $Q[X]$ is controlled by the number of short signed
additive relations among the coefficients (equivalently, by higher-order additive-energy
parameters).

\cref{thm:main} achieves an analogous result in the permutation model
$Y=\sum_{i=1}^n a_i b_{\pi(i)}$ (first introduced by S\"oze \cite{SozeII}, see also \cite{berger2026}). Instead of considering additive
energies, we quantify the
relevant structure of the coefficient multiset $B$ via its multiplicity profile $\mu(B)$ and
the statistic $M(B)$.  Our bound from \cref{thm:main} can be read as stating that the more ``diverse'' $B$ is (as measured by
$M(B)$), the smaller the point masses of the permutation sum must be.  In particular, in the
fully distinct regime one has $M(B)\asymp n^3$, yielding a bound\footnote{While we were preparing this manuscript, this corollary was also independently obtained by Do-Nguyen-Phan-Tran-Vu \cite{DNPTV25}, as well as several other interesting results about this permutation model.} of $n^{-5/2+o(1)}$. On the other hand, in the highly repetitive regimes the estimate is consistent with the
universal $n^{-1}$-type bound of Pawlowski \cite{Pawlo24}.  

We also prove a result where $A$ is allowed to have repeated elements.
\begin{thm} \label{thm:MAMB}
Let $A = \set{a_1,\ldots,a_n}$ and $B = \set{b_1,\ldots,b_n}$ be multisets of real numbers such that $M(A)M(B) \geq n^{3+\eps}$ for some $\eps > 0$. If $\pi$ is a uniformly random permutation of $[n]$, then
\[Q\brac*{\sum_{i=1}^n a_i b_{\pi(i)}} \lsim_\eps \frac{n^{1/2}(\log n)^2}{\sqrt{M(A)M(B)}}.\]
\end{thm}

We note that \cref{thm:MAMB} has the unusual statement that the bound is either relatively strong or nonexistent. 

Let us also remark that \cref{thm:main} and \cref{thm:MAMB} are tight up to subpolynomial factors. Indeed, suppose that $\lambda$ and $\mu$ are partitions of $n$, and $A$ and $B$ are such that $A$ and $B$ contain $\lambda_i$ and $\mu_i$ copies, respectively, of $i-1$ for all $i$. In this case, it is straightforward to compute that
\[\Var\brac*{\sum_{i=1}^n a_i b_{\pi(i)}} \asymp \frac{M(A)M(B)}{n}.\]
Since all elements of $A$ and $B$ are integers, it follows from Chebyshev's inequality and the pigeonhole principle that
\[Q\brac*{\sum_{i=1}^n a_i b_{\pi(i)}} \gsim \frac{n^{1/2}}{\sqrt{M(A)M(B)}}.\]
Moreover, the condition of $M(A)M(B) \geq n^{3+\eps}$ in \cref{thm:MAMB} is necessary. Indeed, if 
\[A = \brac[\big]{\floor{n/2}} \cup \set{0,0,\ldots,0}\quad\text{and}\quad B = \set{1,0,0,\ldots},\]
then $\mu(A) = (\lceil n/2\rceil,1,\ldots,1)$ and $M(A) \asymp \sum_{i=1}^{\lfloor n/2\rfloor} i^2 \asymp n^3$, while $M(B) = 1$. In particular, we have $M(A)M(B) \asymp n^3$. On the other hand, $\sum_{i=1}^n a_i b_{\pi(i)}$ is simply the single value of $A$ matched to the unique nonzero entry of $B$, i.e.\ a uniformly
random element of $A$. Since $A$ contains $\geq n/2$ zeroes, we have
\[Q\brac*{\sum_{i=1}^n a_i b_{\pi(i)}} \geq \setp\left[\sum_{i=1}^n a_i b_{\pi(i)} = 0\right] \geq \frac{1}{2}.\]

\smallskip

We discuss the proofs of \cref{thm:main} and \cref{thm:MAMB} in \cref{sec:proofs}. An important ingredient is a Minkowski distance energy result of Roche-Newton and Rudnev from \cite{RNR15} (which in turn relies on a deep incidence theorem for points and lines in $\mathbb{R}^{3}$ due to Guth and Katz from \cite{guth2015erdos}). 

\medskip
\subsection*{Notation} We use the notation $[N]$ to denote the set $\left\{1,\ldots,N\right\}$. We use $f \lsim g$ to denote $f \leq Cg$ for some $C$, where the constant $C$ may depend on subscripts on the $\lsim$. We use $f \asymp g$ to denote $f \lsim g$ and $f \gsim g$, with subscripts treated similarly.

Given two random variables $X$ and $Y$, we let $Q[X\mid Y]$ denote the random variable $f(Y)$ where $f(y) = \sup_x \setp[X = x \mid Y = y]$. We will often use the fact that $Q[X] \leq \sete_Y[Q[X \mid Y]]$.

\medskip
\subsection*{Acknowledgments}
We thank Noga Alon for helpful conversations. C.P.\ was supported by NSF grant DMS-2246659. D.Z.\ was supported by the NSF Graduate
Research Fellowship Program (grant DGE-2039656).

\section{Proofs} \label{sec:proofs}
We begin with a few probabilistic preliminaries. We first state a standard Chernoff bound.
\begin{fact} \label{fact:chernoff}
Let $n$ be a positive integer and $0 < \delta, p < 1$ be real numbers. If $X$ is the sum of $n$ independent $p$-Bernoulli variables, then
\[\setp[X \leq (1-\delta)pn] \leq e^{-\delta^2 pn/2}.\]
\end{fact}
Via a conditioning argument, this leads to the following corollaries.
\begin{prop} \label{prop:fixsumchernoff}
Let $m \leq n$ be positive integers and $0 < \delta, p < 1$ be real numbers such that $pn$ is an integer. If $S$ is a uniformly random $pn$-element subset of $[n]$, then
\[\setp\brac[\big]{\abs{S \cap [m]} \leq (1-\delta)pm} \leq (n+1)e^{-\delta^2pm/2}.\]
\end{prop}
\begin{proof}
Let $S'$ be a random subset of $[n]$ chosen by sampling every element independently with probability $p$. It is well-known (and easy to check) that the distribution of $\abs{S'}$ peaks at $pn$, so $\setp[\abs{S'} = pn] \geq \frac{1}{n+1}$. Therefore
\begin{multline*}
\setp\brac[\big]{\abs{[m] \cap S} \leq (1-\delta)pm} = \setp\bracmid[\big]{\abs{[m] \cap S'} \leq (1-\delta)pm}{\abs{S'} = pn} \\ = \frac{\setp\brac[\big]{\abs{[m] \cap S'} \leq (1-\delta)pm\text{ and } \abs{S'} = pn}}{\setp\brac[\big]{\abs{S'} = pn}} \leq (n+1) \setp\brac[\big]{\abs{[m] \cap S'} \leq (1-\delta)pm}.
\end{multline*}
Applying \cref{fact:chernoff} finishes.
\end{proof}
\begin{prop} \label{prop:pairchernoff}
Let $m, n$ be positive integers such that $2m \leq n$ and $0 < \delta, p, q < 1$ be real numbers such that $pn$ and $qn$ are integers. Let $S$ and $T$ be uniformly random $pn$- and $qn$-element subsets of $[n]$, respectively, chosen independently from each other. Then, if $I$ is the set of indices $i \in [m]$ such that
\[\abs{\set{2i-1, 2i} \cap S} = \abs{\set{2i-1, 2i} \cap T} = 1,\]
we have
\[\setp\brac[\big]{\abs{I} \leq (1-\delta)\cdot 4p(1-p)q(1-q)m} \leq (n+1)^2e^{-\delta^2\cdot 4p(1-p)q(1-q)m/2}.\]
\end{prop}
\begin{proof}
Let $S'$ and $T'$ be random subsets of $[n]$ chosen by sampling every element independently with probabilities $p$ and $q$, respectively. Define $I'$ similarly to $I$ but with $S$ and $T$ replaced with $S'$ and $T'$. Arguing similarly to the proof of \cref{prop:fixsumchernoff}, it suffices to show that
\[\setp\brac[\big]{\abs{I'} \leq (1-\delta)\cdot 4p(1-p)q(1-q)m} \leq e^{-\delta^2 \cdot 4p(1-p)q(1-q)m/2}.\]
To see this, observe that every index $i \in [m]$ is in $I'$ with probability $4p(1-p)q(1-q)$, and that these events are independent. Thus we are done by \cref{fact:chernoff}.
\end{proof}

We now introduce a notion of additive structure which will be useful for the proof.

\begin{defn}
Let $s \geq 2$ be an integer. Suppose $A = \set{a_1,\ldots,a_n}$ and $B = \set{b_1,\ldots,b_{n'}}$ are nonempty multisets of real numbers, and $c = (c_1,\ldots,c_s)$ is a tuple of nonzero integers. Let $K_c(A,B)$ is the number of $4s$-tuples $(i_1,j_1,\ldots,i_s,j_s,i'_1,j'_1,\ldots,i'_s,j'_s) \in [n]^{2s} \times [n']^{2s}$ such that
\[\sum_{k=1}^s c_k (a_{i_k} - a_{j_k})(b_{i'_k} - b_{j'_k})  = 0.\]
Let $K'_c(A,B)$ be the number of such tuples, where we impose the additional condition that $i_1,j_1,\ldots,i_s,j_s$ are distinct and $i'_1,j'_1,\ldots,i'_s,j'_s$ are distinct.
Finally, let 
\[\kappa_c(A,B) = K_c(A,B)/(nn')^{2s}\quad\text{and}\quad\kappa'_c(A,B) = K_c'(A,B)/(nn')^{2s}.\]
\end{defn}
We first show an anticoncentration result in terms of the $\kappa'_c(A, B)$.
\begin{lem} \label{lem:halasz}
Let $A = \set{a_1,\ldots,a_n}$ and $B = \set{b_1,\ldots,b_n}$ be multisets of real numbers, neither of which contain more than $2n/3$ copies of a single element. Let $\pi$ be a uniformly random permutation of $[n]$. Then for any $D > 0$ there exists a finite set $C$ of tuples of nonzero integers (depending only on $D$) such that
\[Q\brac*{\sum_{i=1}^n a_i b_{\pi(i)}} \lsim_D n^{-D} + \frac{\sum_{c\in C} \kappa'_c(A,B)}{n^{1/2}}.\]
\end{lem}
\begin{proof}
To prove Lemma \ref{lem:halasz}, we will use the aforementioned anticoncentration bound of Hal\'asz \cite{Hal77}:
\begin{thm}[Hal\'asz] \label{thm:halasz}
Let $a_1,\ldots,a_n$ be nonzero reals and let $\xi_1,\ldots,\xi_n$ be independent random variables drawn from $\set{0,1}$ uniformly at random. Then
\[Q\brac*{\sum_{i=1}^n a_i \xi_i} \lsim_r R_r n^{-2r-1/2},\]
where for a positive integer $r$, we define $R_r$ to be the number of choices $(i_1,\ldots,i_{2r}) \in [n]^{2r}$ and $(\eps_1,\ldots,\eps_{2r}) \in \set{\pm 1}^{2r}$ such that $\eps_1 a_{i_1} + \cdots + \eps_{2r} a_{i_{2r}} = 0$.
\end{thm}
To apply this theorem, let $m = \floor{n/2}$, let $\pi'$ be a uniformly random permutation independent from $\pi$, and let $\sigma$ be a random permutation obtained by taking the product of a uniformly random subset of the $m$ transpositions $(1\;2), (3\;4),\ldots,(2m-1\;2m)$. Then $\sum_i a_{\pi(i)} b_{\pi'(\sigma(i))}$ has the same distribution as $\sum_i a_i b_{\pi(i)}$, so we have
\[Q\brac*{\sum_{i=1}^n a_i b_{\pi(i)}} \leq \sete_{\pi,\pi'}\brac*{Q\bracmid*{\sum_{i=1}^n a_{\pi(i)} b_{\pi'(\sigma(i))}}{\pi,\pi'}}.\]
Letting $\xi_i$ be the indicator variable of the occurrence of $(2i-1\;2i)$ in $\sigma$, we have
\[\sum_{i=1}^n a_{\pi(i)} b_{\pi'(\sigma(i))} = \sum_{i=1}^n a_{\pi(i)} b_{\pi'(i)} - \sum_{i=1}^m \xi_i(a_{\pi(2i-1)} - a_{\pi(2i)})(b_{\pi'(2i-1)} - b_{\pi'(2i)}).\]
In the case where there there exist at least $m/10$ values of $i \in [m]$ such that $(a_{\pi(2i-1)} - a_{\pi(2i)})(b_{\pi'(2i-1)} - b_{\pi'(2i)}) \neq 0$ (call $(\pi,\pi')$ \vocab{good} if this is true), an application of \cref{thm:halasz} (ignoring the zero coefficients) implies that
\[Q\bracmid*{\sum_{i=1}^n a_i b_{\pi(i)}}{\pi,\pi'} \lsim_r R_{r,\pi,\pi'} n^{-2r-1/2},\]
where $R_{r,\pi,\pi'}$ is the number of choices of $(i_1,\ldots,i_{2r}) \in [m]^{2r}$ and $(\eps_1,\ldots,\eps_{2r}) \in \set{\pm 1}^{2r}$ such that
\[\sum_{j=1}^{2r} \eps_j (a_{\pi(2i_j-1)} - a_{\pi(2i_j)})(b_{\pi'(2i_j-1)} - b_{\pi'(2i_j)}) = 0. \tag{$\divideontimes$} \label{eq:rrpp}\]
We conclude that
\[Q\brac*{\sum_{i=1}^n a_i b_{\pi(i)}} \lsim_r \setp[(\pi,\pi')\text{ is not good}] + \sete[R_{r,\pi,\pi'}] n^{-2r-1/2}.\]

To bound the probability that $(\pi,\pi')$ is not good, we first claim that we may color the elements of $A$ and $B$ red and blue such that no two occurrences of the same element in the same multiset receive different colors, and that the counts of red elements in $A$ and $B$ are both in the interval $[n/3, 2n/3]$. In the case where an element of $A$ occurs with multiplicity at least $n/3$, we may simply color all occurrences of this element red and all other elements blue. If not, consider initially coloring all elements of $A$ blue and iteratively choosing a blue element of $A$ and coloring all of its occurrences in $A$ red. Since all multiplicities are less than $n/3$, at some point the number of red elements must be in $[n/3, 2n/3]$. Repeating the same argument for $B$ yields the desired coloring.

Given this coloring, the set of indices $i \in [n]$ such that $a_{\pi(i)}$ is red has the same distribution as a random $pn$-element subset of $[n]$, for some $p$ with $p(1-p) \geq 2/9$. The set of indices $i \in [n]$ such that $b_{\pi'(i)}$ is red has the same distribution as an independent random $qn$-element subset of $[n]$, for some $q$ with $q(1-q) \geq 2/9$. Since $4 \cdot \frac{2}{9} \cdot \frac{2}{9} = \frac{16}{81} > \frac{1}{10}$, applying \cref{prop:pairchernoff} yields that there exists some absolute constant $\alpha > 0$ such that with probability at least $1-(n+1)^2e^{-\alpha m}$, there are at least $m/10$ indices $i \in [m]$ such that $a_{\pi(2i-1)}$ and $a_{\pi(2i)}$ receive different colors and $b_{\pi'(2i-1)}$ and $b_{\pi'(2i)}$ receive different colors. In this case, $(\pi, \pi')$ must certainly be good, so
\[\setp[(\pi,\pi') \text{ is not good}] \leq (n+1)^2 e^{-\alpha m}.\]
for some $\alpha > 0$.

To bound $\sete[R_{r,\pi,\pi'}]$, fix a choice of $(i_1,\ldots,i_{2r}) \in [m]^{2r}$ and $(\eps_1,\ldots,\eps_{2r}) \in \set{\pm 1}^{2r}$ and consider combining like terms in \labelcref{eq:rrpp}. It is possible that we are left with one or fewer terms, i.e.\ $\sum_{i_j = i} \eps_j \neq 0$ for at most one choice of $i$. There are $O_r(n^r)$ cases where this occurs, since if $r+1$ distinct values of $i$ occur, at least two must appear exactly once. For all other choices, the probability that \labelcref{eq:rrpp} is satisfied is exactly
\[\frac{K'_c(A, B)}{(n(n-1) \cdots (n-2s+1))^2} \lsim_r \kappa'_c(A, B)\]
for some $c$ of length $s \geq 2$. Since there are only finitely many possibilities for $c$, we conclude that there is a finite set $C$ such that
\[\sete[R_{r,\pi,\pi'}] \lsim_r n^r + n^{2r}\sum_{c\in C}\kappa'_c(A,B).\]
Putting everything together, we get
\[Q\brac*{\sum_{i=1}^n a_i b_{\pi(i)}} \lsim_r (n+1)^2 e^{-\alpha m} + n^{-r-1/2} + \frac{\sum_{c\in C}\kappa'_c(A,B)}{n^{1/2}}.\]
Choosing $r \geq D-1/2$ finishes.
\end{proof}

The next lemma is a uniform estimate for every $\kappa_c(A, B)$, with input from incidence geometry. 

\begin{lem} \label{lem:rnr}
Let $A$ and $B$ be finite sets of real numbers of size at least $2$. Then for any $c$,
\[\kappa_c(A, B) \lsim \frac{\log {\abs{A}} + \log {\abs{B}}}{\abs{A}\abs{B}}.\]
\end{lem}
\begin{proof}
It is a result of Roche-Newton and Rudnev \cite[Prop.~4]{RNR15} that
\[\kappa_{1,-1}(A, B) \lsim \frac{\log {\abs{A}} + \log {\abs{B}}}{\abs{A}\abs{B}},\]
so it suffices to show that $\kappa_c(A, B) \leq \kappa_{1,-1}(A,B)$ for all $c$. To do this, let $A_1,A_2$ be independent random elements of $A$ and $B_1,B_2$ be independent random elements of $B$. Let $Z = (A_1 - A_2)(B_1 - B_2)$ and let $Z_1,Z_2,\ldots$ be independent copies of $Z$. Note that
\[\kappa_c(A, B) = \setp[c_1Z_1 + \cdots + c_sZ_s = 0].\]
By conditioning on $Z_3,Z_4,\ldots,Z_s$, we find that
\[\kappa_c(A, B) \leq Q[c_1Z_1 + c_2Z_2] = \sup_a \sum_{c_1x+c_2y=a} \setp[Z = x] \setp[Z = y].\]
By the Cauchy-Schwarz inequality, this is bounded above by
\[\sum_x \setp[Z = x]^2 = \setp[Z_1 = Z_2] = \kappa_{1,-1}(A,B),\]
as desired.
\end{proof}
\begin{lem} \label{lem:mmrr}
Suppose $A$ contains as a submultiset $m$ copies of a set $A'$ of size $r$, and $B$ contains as a submultiset $m'$ copies of a set $B'$ of size $r'$. Further assume that $r,r' \geq 2$ and $mm'rr' \geq n^{1+\eps}$ for some $\eps > 0$. Then,
\[Q\brac*{\sum_{i=1}^n a_i b_{\pi(i)}} \lsim_{\eps} \frac{n^{1/2}\log n}{(mm')^{1/2}(rr')^{3/2}}.\]
\end{lem}
\begin{proof}
Without loss of generality suppose the $m$ copies of $A'$ are the first $mr$ elements of $A$ and $m'$ copies of $B'$ are the first $m'r'$ elements of $B$. Then consider sampling $\pi$ according to the following process:
\begin{enumerate}
    \item Choosing $u = \abs{\pi([mr]) \cap [m'r']}$ according to the distribution of $\abs{S \cap [m'r']}$, where $S$ is a random $mr$-element subset of $[n]$.
    \item Choosing $I = [mr] \cap \pi\inv([m'r'])$ and $I' = \pi([mr]) \cap [m'r']$ to be uniformly random $u$-element subsets of $[mr]$ and $[m'r']$, respectively.
    \item Assigning $\pi\rvert_{[n]\setminus I}$ according to some distribution.
    \item Choosing $\pi\rvert_I \colon I \to I'$ to be a uniformly random bijection.
\end{enumerate}
By \cref{prop:fixsumchernoff}, we have
\[\setp\brac*{u \leq \frac{mrm'r'}{2n}} \leq (n+1) e
^{mrm'r'/(8n)} \leq (n+1)e^{-n^\eps/8}.\]
Henceforth condition on a value of $u > mrm'r'/(2n) \geq n^\eps/2$.

Consider an arbitrary $a \in A'$ and let $J$ be the indices $j \in [mr]$ such that $a_j \neq a$. Then by \cref{prop:fixsumchernoff} again we have
\[\setp\brac*{\abs{I \cap J} \leq \frac{2\abs{J}}{3mr}u} \leq (mr+1)e^{-\frac{1}{18} u\abs{J}/(mr)}.\]
Since $\abs{J} \geq mr/2$, we conclude that
\[\setp\brac[\big]{\abs{I \cap J} \leq u/3} \leq (n+1)e^{-u/36} \leq (n+1)e^{-n^\eps/72}.\]
Union bounding over all $a$, we find that with probability at least $1-n(n+1)e^{-n^\eps/72}$, the multiset $\setmid{a_i}{i \in I}$ does not contain any element more than $2u/3$ times. An identical argument shows the same for $\setmid{b_i}{i \in I'}$. Henceforth condition on choices of $I$ and $I'$ such that this is true.

Under these assumptions, the randomness of step (4) is exactly the situation of \cref{lem:halasz}, which yields
\[Q\bracmid*{\sum_{i=1}^n a_i b_{\pi(i)}}{u,I,I',\pi\rvert_{[n]\setminus I}} \lsim_D u^{-D} + \frac{\sum_{c \in C}\kappa'_c(\setmid{a_i}{i \in I},\setmid{b_i}{i \in I'})}{u^{1/2}}.\]
For $c$ of length $s$, we have
\begin{align*}
    \sete_{I,I'}\brac[\big]{\kappa'_c\paren[\big]{\setmid{a_i}{i \in I},\setmid{b_i}{i \in I'}}} &= \frac{1}{u^{4s}} \frac{\binom{u}{2s}^2}{\binom{mr}{2s}\binom{m'r'}{2s}} K'_c\paren[\big]{\setmid{a_i}{i \in [ms]},\setmid{b_i}{i \in [m'r']}} \\ 
    &\leq \frac{1}{u^{4s}} \frac{u^{4s}}{(mrm'r')^{2s}} K_c\paren*{\setmid{a_i}{i \in [ms]},\setmid{b_i}{i \in [m'r']}} \\
    &= \kappa_c(A', B') \lsim \frac{\log n}{rr'},
\end{align*}
where we have used \cref{lem:rnr}. Therefore for $u > mrm'r'/(2n)$ we have
\[Q\bracmid*{\sum_{i=1}^n a_i b_{\pi(i)}}{u} \lsim_{D} n(n+1)e^{-n^\eps/72} + u^{-D} + \frac{\log n}{u^{1/2} rr'}.\]
This in turn implies
\[Q\brac*{\sum_{i=1}^n a_i b_{\pi(i)}} \lsim_D (n+1)e^{-n^\eps/8} + n(n+1)e^{-n^\eps/72} + n^{-D\eps} + \frac{n \log n}{(mm')^{1/2} (rr')^{3/2}}.\]
By choosing $D \geq 10/\eps$, the first three terms are $O_\eps(n^{-10})$, which is insignificant compared to the final term. This concludes the proof.
\end{proof}

Given \cref{lem:mmrr}, \cref{thm:main,thm:MAMB} follow from a short dyadic partitioning argument.

\begin{prop} \label{prop:partition}
Every multiset $A$ of size $n$ with at least two distinct elements contains, as a submultiset, $m$ copies of a set of size $r$ for some $m,r$ with $r \geq 2$ and $mr^3 \geq M(A)/\log n$.
\end{prop}
\begin{proof}
Let $\mu(A) = (\mu_1,\ldots,\mu_\ell)$. If there exists some $i \geq 2$ with $i^3\mu_i \geq M(A)/\log n$ we would be done. If not, then
\[M(A) = \sum_{i=2}^\ell(i-1)^2\mu_i < \frac{M(A)}{\log n} \sum_{i=2}^\ell \frac{(i-1)^2}{i^3} < \frac{M(A)}{\log n} \sum_{i=2}^n \frac{1}{i} < \frac{M(A)}{\log n} \int_1^n \frac{dt}{t} = M(A),\]
a contradiction.
\end{proof}

\begin{proof}[Proof of \cref{thm:MAMB}]
By \cref{prop:partition}, we may choose $m,r,m',r'$ with $r,r' \geq 2$ and $mm'(rr')^3 \geq M(A)M(B)/(\log n)^2$ such that $A$ contains $m$ copies of a set of size $r$ and $B$ contains $m'$ copies of a set of size $r'$. If $M(A)M(B) \geq n^{3+\eps}$, then we additionally have
\[mm'rr' \geq (mm')^{1/3}rr' \geq \frac{n^{1+\eps/3}}{(\log n)^{2/3}},\]
so \cref{lem:mmrr} applies with $\eps$ replaced with $\eps/4$ for large $n$. We conclude that
\[Q\brac*{\sum_{i=1}^n a_i b_{\pi(i)}} \lsim_\eps \frac{n^{1/2}\log n}{\sqrt{mm'(rr')^3}} \leq \frac{n^{1/2}(\log n)^2}{\sqrt{M(A)M(B)}}.\]
\end{proof}

\begin{proof}[Proof of \cref{thm:main}]
Pick an $\eps > 0$. If $M(A)M(B) \geq n^{3+\eps}$, applying \cref{thm:MAMB} yields that
\[Q\brac*{\sum_{i=1}^n a_i b_{\pi(i)}} \lsim_\eps \frac{n^{1/2}(\log n)^2}{\sqrt{M(A)M(B)}} \lsim_\eps \frac{n^{-1+\eps}}{\sqrt{M(B)}}.\]
If $M(A)M(B) < n^{3+\eps}$ then we have $M(B) \lsim n^\eps$, so \cref{thm:paw} yields
\[Q\brac*{\sum_{i=1}^n a_i b_{\pi(i)}} \lsim n^{-1} \lsim \frac{n^{-1+\eps}}{\sqrt{M(B)}}.\]
Since $\eps$ was arbitrary, we get the desired bound.
\end{proof}

\section{Concluding Remarks}

We previously highlighted that when $B$ is a set, $M(B)\asymp n^3$, and \cref{thm:main} specializes to the near-optimal bound
$Q \lsim n^{-5/2+o(1)}$, an analogue of a theorem of Erd\H{o}s and Moser. It is therefore natural to ask whether one can take this analogy even one step further and remove the $n^{o(1)}$ loss completely, in the spirit of the S\'ark\"{o}zy-Szemer\'edi and Stanley results for the classical Littlewood-Offord model. In this direction, we make the following precise conjecture:

\begin{conj}\label{conj:AP-extremal}
Let $A=\{a_1,\dots,a_n\}$ and $B=\{b_1,\dots,b_n\}$ be \emph{sets} of $n$ distinct real numbers, and
let $\pi$ be uniform in $S_n$. Then
\[
Q\brac*{\sum_{i=1}^n a_i\,b_{\pi(i)}}\ \le\ \Big(\frac{12}{\sqrt{2\pi}}+o(1)\Big)\,n^{-5/2}.
\]
\end{conj}

The constant $\frac{12}{\sqrt{2\pi}}$ comes from the case $A=B=[n]$, which we believe should be extremal. Proving \cref{conj:AP-extremal} would require new input beyond our current approach. For example, one of the reasons behind the $n^{o(1)}$ loss in \cref{thm:main} is the Minkowski distance energy bound of Roche-Newton and Rudnev. This theorem itself contains a logarithmic loss which cannot be removed due to grid examples (an observation which is tied to the classical multiplication-table phenomenon). That being said, the conditioning trick from the proof of \cref{lem:rnr} is flexible enough to allow the replacement of \cite[Prop.~4]{RNR15} with appropriate higher order energy estimates that may not require logarithmic losses. It would already be interesting to find such an estimate and to use it to prove an $O(n^{-5/2})$ bound for \cref{conj:AP-extremal}. To get the precise leading constant $\frac{12}{\sqrt{2\pi}}$, however, we suspect that completely different ideas would be required.

In a different direction, we also note that unlike Hal\'asz's theorem (\cref{thm:halasz}), which can prove arbitrarily strong bounds by considering arbitrarily large $r$, \cref{thm:main,thm:MAMB} cannot prove a bound stronger than $n^{-5/2+o(1)}$. It would be interesting to see if there exist strengthenings of \cref{thm:main,thm:MAMB} that prove arbitrarily strong bounds as the additive structure of $A$ and $B$ decrease. After all, for generic $A$ and $B$, the map $\pi \mapsto \sum_{i} a_i b_{\pi(i)}$ is injective and thus 
\[Q\brac*{\sum_{i=1}^{n} a_i b_{\pi(i)}} = \frac{1}{n!}.\]
This shows that, in principle, permutation sums can exhibit even factorial anticoncentration once the sets of coefficients are sufficiently additively independent. 

Such a result would have to look beyond multiplicities for quantifying additive structure, but it is unclear to us what the correct notion would be.

\printbibliography
\end{document}